\documentclass[12pt,reqno]{amsart}

\usepackage{amssymb,amsthm}
\usepackage{amsfonts,cmtiup,comment,stmaryrd}

\usepackage{mathrsfs,cmtiup}

\makeatletter
\def\blfootnote{\xdef\@thefnmark{}\@footnotetext}
\makeatother

\newtheorem{theorem}{Theorem}[section]

\newtheorem{lemma}[theorem]{Lemma}
\newtheorem{proposition}[theorem]{Proposition}
\newtheorem{corollary}[theorem]{Corollary}

\newtheorem{hyp}[theorem]{Hypothesis}
\newtheorem{conj}[theorem]{Conjecture}

\theoremstyle{definition}

\newtheorem{remark}[theorem]{Remark}
\newtheorem*{definition*}{Definition}

\newcommand{\g}{\gamma}
\newcommand{\F}{{\Bbb F}}

\newcommand{\bt}{\begin{theorem}}
\newcommand{\et}{\end{theorem}}
\newcommand{\bc}{\begin{corollary}}
\newcommand{\ec}{\end{corollary}}
\newcommand{\bpr}{\begin{proposition}}
\newcommand{\epr}{\end{proposition}}
\newcommand{\be}{\begin{equation}}
\newcommand{\ee}{\end{equation}}
\newcommand{\bp}{\begin{proof}}
\newcommand{\ep}{\end{proof}}
\newcommand{\bconj}{\begin{conj}}
\newcommand{\econj}{\end{conj}}
\newcommand{\bl}{\begin{lemma}}
\newcommand{\el}{\end{lemma}}
\newcommand{\bh}{\begin{hyp}}
\newcommand{\eh}{\end{hyp}}
\newcommand{\br}{\begin{remark}}
\newcommand{\er}{\end{remark}}

\let\leq=\leqslant
\let\geq=\geqslant

\let\leq=\leqslant
\let\geq=\geqslant
\numberwithin{equation}{section}
\newcommand{\ed}{\end{document}}

\begin{document}
\title{Finite groups with Engel sinks of bounded rank}

\author{E. I. Khukhro}
\address{University of Lincoln, U.K., and\newline
Sobolev Institute of Mathematics, Novosibirsk, 630090, Russia}
\email{khukhro@yahoo.co.uk}

\author{P. Shumyatsky}

\address{Department of Mathematics, University of Brasilia, DF~70910-900, Brazil}
\email{pavel@unb.br}



\begin{abstract}
For an element $g$ of a group $G$, an Engel sink is a subset ${\mathscr E}(g)$ such that for every $x\in G$ all sufficiently long commutators $[...[[x,g],g],\dots ,g]$ belong to ${\mathscr E}(g)$. A~finite group is nilpotent if and only if every element has a trivial Engel sink. We prove that if in a finite group $G$ every element has an Engel sink generating a subgroup of rank~$r$, then
$G$ has a normal subgroup $N$ of rank bounded in terms of $r$ such that $G/N$ is nilpotent.
\end{abstract}
\maketitle

\section*{Introduction}\label{s-intro}
\baselineskip18pt

 A~group $G$ is called an \emph{Engel group} if for every $x,g\in G$ the equation $[x,g,g,\dots , g]=1$ holds, where $g$ is repeated in the commutator sufficiently many times depending on $x$ and $g$. (Throughout the paper, we use the left-normed simple commutator notation
$[a_1,a_2,a_3,\dots ,a_r]=[...[[a_1,a_2],a_3],\dots ,a_r]$.)
Of course, any locally nilpotent group is an Engel group. In some classes of groups the converse is also known to be true. For example, a finite Engel group is nilpotent by Zorn's theorem \cite{zorn}. Wilson and Zelmanov~\cite{wi-ze} proved that profinite Engel groups are locally nilpotent, and Medvedev~\cite{med} extended this result to compact (Hausdorff) groups.

As a next step, it is natural to consider groups that are `almost Engel' in the sense of restrictions on so-called Engel sinks. An Engel sink of an element $g\in G$ is a
set ${\mathscr E}(g)$ such that for every $x\in G$ all sufficiently long commutators $[x,g,g,\dots ,g]$ belong to ${\mathscr E}(g)$, that is, for every $x\in G$ there is a positive integer $n(x,g)$ such that
 $$[x,\underbrace{g,g,\dots ,g}_n]\in {\mathscr E}(g)\qquad \text{for all }n\geq n(x,g).
 $$
Engel groups are precisely the groups for which we can choose ${\mathscr E}(g)=\{ 1\}$ for all $g\in G$. In \cite{khu-shu162} we considered finite, profinite, and compact groups in which every element has a finite Engel sink.
 We proved in \cite{khu-shu162} that compact groups with this property are finite-by-(locally nilpotent).

 Results for finite groups have to be of quantitative nature. Obviously, in a finite group every element has the smallest Engel sink, so from now on we use the term Engel sink for the minimal Engel sink of $g$, denoted by ${\mathscr E}(g)$, thus eliminating ambiguity in this notation. We proved in \cite[Theorem~3.1]{khu-shu162} that if $G$ is a finite group and there is a positive integer $m$ such that $|{\mathscr E}(g)|\leq m$ for all $g\in G$, then $G$ has a normal subgroup $N$  of order bounded in terms of $m$ such that the quotient $G/N$ is nilpotent.

In this paper we consider finite groups in which there is a bound for the rank of the subgroups generated by the Engel sinks. Here, the \textit{rank} of a finite group is the minimum number $r$ such that every subgroup can be generated by $r$ elements.

\bt
\label{t-main}
Suppose that $G$ is a finite group such that for every $g\in G$ the Engel sink ${\mathscr E}(g)$ generates a subgroup of rank $r$. Then $G$ has a normal subgroup $N$ of rank bounded in terms of $r$ such that the quotient $G/N$ is nilpotent.
\et

Clearly, the conclusion can also be stated as ``Then the rank of the nilpotent residual $\g _{\infty}(G)$ is bounded in terms of $r$.'' Here, $\g _{\infty}(G)=\bigcap_i\g_i(G)$ is the intersection of all terms of the lower central series.

First we prove the theorem for soluble finite groups. Then we consider the nonsoluble case (where we use the classification of finite simple groups).

\section{Preliminaries}
\label{s-prel}

The following result was obtained by Kov\'acs~\cite{kov} for
soluble groups, and extended independently by Guralnick \cite{gur} and Lucchini \cite{luc} using the classification (improving a bound $2d$ of Longobardi and Maj \cite{lo-ma}).

\begin{lemma}\label{l-kov} If $d$ is the maximum of the ranks of
the Sylow subgroups of a~finite group, then the rank of this
group is at most~$d+1$.
\end{lemma}

The following lemma appeared independently and simultaneously in
the papers of Gorchakov~\cite{grc}, Merzlyakov~\cite{me}, and as
``P.\,Hall's lemma" in the paper of Roseblade~\cite{rs}.

\begin{lemma}\label{l-gmh}
 Let $p$ be a~prime number. The rank of a~$p$-group of
automorphisms of an abelian 
finite $p$-group of rank~$r$ is bounded in
terms of~$r$.
\end{lemma}

The next two lemmas must also be well known. For brevity we say that a quantity is \textit{$a$-bounded} if it is bounded above in terms of a parameter $a$.

\begin{lemma}\label{l-r-coprime}
 A~finite $p'$-group $Q$ of linear
transformations of a vector space of dimension $n$ over a field of
characteristic $p$ has $n$-bounded rank.
\end{lemma}

\begin{proof} By Lemma~\ref{l-kov} we can assume that $Q$ is a $q$-group for a prime $q$. Choose a maximal abelian normal subgroup $A$ in $Q$. Since $Q/A$ acts faithfully on $A$, by Lemma~\ref{l-gmh} it
suffices to bound the rank of~$A$. After extension of the field,
$A$ is diagonalizable, which gives the result since finite
multiplicative subgroups of fields are cyclic. \end{proof}

Let $F(G)$ denote the Fitting subgroup of a group $G$, the largest normal nilpotent subgroup. The Fitting series starts with $F_1(G)=F(G)$, and then by induction $F_{i+1}(G)$ is the inverse image of $F(G/F_i(G))$. The Fitting height of a soluble finite group $G$ is the minimum $h$ such that $F_h(G)=G$. We also use the usual notation $O_{p'}(G)$ for the largest normal $p'$-subgroup of $G$, and $O_{p',p}(G)$ for the inverse image of the largest normal $p$-subgroup of $G/O_{p'}(G)$.

\begin{lemma}\label{l-r-fh}
 A~finite soluble group $G$ of rank $r$
has $r$-bounded Fitting height.
 \end{lemma}

\begin{proof} For every prime $p$ the quotient $G/O_{p',p}(G)$ acts faithfully
on the Frattini quotient of $O_{p',p}(G)/O_{p'}(G)$ and therefore
is a linear group of dimension at most~$r$. By 
Zassenhaus' theorem \cite[15.1.3]{rob}, the derived length of $G/O_{p',p}(G)$ is
 $r$-bounded . Hence the same is true
for $G/F(G)=G/\bigcap _{p}O_{p',p}(G)$. \end{proof}

The next technical lemma is also a well-known fact (see, for example, \cite[Lemma~10]{khu-maz}).

\bl \label{l-f2}
 For any finite group $H=F_2(H)$ of Fitting
height~$2$ we have
$$\gamma _{\infty}(H)=\prod
_{q}[F_q,\,H_{q'}],$$
where $F_q$ is the Sylow $q$-sub\-group of~$F(H)$ and $H_{q'}$ is a
Hall $q'$-sub\-group of~$H$.
\el

The following elementary lemma will be used several times.

\begin{lemma}\label{l-prod} Suppose that a group $A$ acts by automorphisms on a group $G$. If $A=\langle a_1,\dots ,a _k\rangle$, then
$[G,A]=[G,a_1]\cdots [G,a _k]$.
\end{lemma}

\begin{proof}
The product $[G,a_1]\cdots [G,a _k]$ is a normal subgroup of $G$. This product is also $A$-invariant,
since it is invariant under every generator $a_i$ of $A$:
$$
\big[[G,a_1]\cdots [G,a _k],\,a_i\big]\leq [G,a_i]\leq [G,a_1]\cdots [G,a _k].
$$
Furthermore, $A$ acts trivially on the quotient by this product, since so does every generator of $A$. Hence, $[G,A]\leq [G,a_1]\cdots [G,a _k]$. The reverse inclusion is obvious.
\ep

The following lemma relates Engel sinks in finite groups to coprime actions.
We denote the derived subgroup of a group $X$ by $X'$.

\begin{lemma}\label{l0}
Let $P$ be a finite $p$-subgroup of a group $G$, and $g\in G$ a $p'$-element normalizing $P$. Then $[P,g]\leq \langle {\mathscr E}(g)\rangle$.
\end{lemma}

\begin{proof}
For the abelian $p$-group $V=[P,g]/[P,g]'$ we have $V=[V,g]$ and $C_V(g)=1$ because the action of $g$ on $V$ is coprime. Then $V=\{[v,g]\mid v\in V\}$ and therefore also
$$
V=\{[v,\underbrace{g,\dots ,g}_n\,]\mid v\in V\}
$$
for any $n$. Hence, $V$ is contained in the image of ${\mathscr E}(g)\cap [P,g]$ in $[P,g]/[P,g]'$, whence the result.
\ep

\section{Soluble groups}\label{s-sol}

Throughout what follows, let ${\bf r}(X)$ denote the rank of a finite group~$X$.

\bpr
\label{pr1}
Let $q$ be a prime, let $Q$ be a finite $q$-group, and $U$ a $q'$-group of
automorphisms of $Q$. Suppose that ${\bf r}([Q,u])\leq r$ for every $u\in U$. Then
 \begin{itemize}
 \item[\rm (a)] ${\bf r}(U)$ is $r$-bounded;
 \item[\rm (b)] ${\bf r}([Q,U])$ is $r$-bounded;
 \item[\rm (c)] if $U$ is soluble, then the derived length of $U$ is $r$-bounded.
 \end{itemize}

\epr

\begin{proof}
(a) First suppose that $U$ is abelian. We consider the Frattini quotient $V=Q/\Phi (Q)$ as a faithful ${\mathbb F} _qU$-module. Pick $u_1\in U$ such that $[V,u_1]\ne 0$. By Maschke's theorem, $V= [V,u_1]\oplus C_V(u_1)$, where both summands are $U$-invariant, since $U$ is abelian. If $C_U([V,u_1])=1$, then ${\bf r}(U)$ is $r$-bounded by Lemma~\ref{l-r-coprime}. Otherwise pick $1\ne u_2\in C_U([V,u_1])$; then $V= [V,u_1] \oplus [V,u_2] \oplus C_V(\langle u_1,u_2\rangle )$. If $1\ne u_3\in C_U([V,u_1]\oplus [V,u_2])$, then $V= [V,u_1]\oplus [V,u_2]\oplus [V,u_3] \oplus C_V(\langle u_1,u_2,u_3\rangle )$, and so on. If $C_U([V,u_1]\oplus\dots \oplus [V,u_k])=1$ at some step $k\leq r$, then again ${\bf r}(U)$ is $r$-bounded by Lemma~\ref{l-r-coprime}. However, if there are too many steps, say, $k$ steps for $k>r$, then for the element $w=u_1u_2\cdots u_k$ we shall have $0\ne [V,u_i]= [[V,u_i],w]$, so that $[V,w] = [V,u_1]\oplus\dots \oplus [V,u_k]$ will have rank greater than $r$, a contradiction.

We now consider the general case. Let $P$ be a Sylow $p$-subgroup of $U$, and $M$ a maximal normal abelian subgroup of $P$. By the above, ${\bf r}(M)$ is $r$-bounded. Then ${\bf r}(P)$ is $r$-bounded by Lemma~\ref{l-gmh}, since $P/M$ acts faithfully on $M$. Thus, the rank of a Sylow $p$-subgroup of $U$ is $r$-bounded for every $p$, which implies that the rank of $U$ is $r$-bounded by Lemma~\ref{l-kov}.

(b) By part (a), in particular, $U=\langle a_1,\dots ,a_f\rangle$ for some $r$-bounded $f$. Then $[Q,U]=[Q,a_1]\cdots [Q,a_f]$ by Lemma~\ref{l-prod}. Since each of the normal subgroups $[Q,a_i]$ has rank at most $r$ and $f$ is $r$-bounded, the rank of $[Q,U]$ is also $r$-bounded.

(c) The group $U$ acts faithfully on the Frattini quotient $[Q,U]/\Phi ([Q,U])$, which can be regarded as a vector space over $\F _q$, the dimension of which is $r$-bounded by part (b). If $U$ is soluble, then its derived length is $r$-bounded by Zassenhaus' theorem \cite[15.1.3]{rob}.
\ep

\bt\label{t-sol}
Suppose that $G$ is a finite soluble group such that for every $g\in G$ the Engel sink ${\mathscr E}(g)$ generates a subgroup of rank $r$. Then ${\bf r}(\g _{\infty}(G))$ is $r$-bounded.
\et

\bp
Note that the hypothesis is inherited by all sections. By Lemma~\ref{l0}, the hypothesis implies that for any $p$-subgroup $P$ of any section of $G$ and a $p'$-element $g$ of this section normalizing $P$, the rank of $[P,g]$ is at most $r$.

First we prove that the Fitting height $h(G)$ is $r$-bounded.
Since $F(G)=\bigcap_{p}O_{p',p}(G)$, it is sufficient to bound the Fitting height of each quotient $\bar G=G/O_{p',p}(G)$. Since a Hall $p'$-subgroup of $\bar G$ acts faithfully on the Frattini quotient of $O_{p',p}(G)/O_{p'}(G)$, the rank of a Hall $p'$-subgroup of $\bar G$ is $r$-bounded by Proposition~\ref{pr1}(a). Then the Fitting height of a Hall $p'$-subgroup of $\bar G$ is also $r$-bounded by Lemma~\ref{l-r-fh}. Therefore, in order to bound the Fitting height $h(\bar G)$, it remains to bound the $p$-length of $\bar G$.

Let $P$ be a Sylow $p$-subgroup of $\bar G$. Then $P$ acts faithfully on $F(\bar G)=Q_1\times\cdots \times Q_s$, where $Q_i$ is a $q_i$-subgroup and $q_i\ne p$ for all $i$. For every $i$, the derived length of $P/C_P(Q_i)$ is $r$-bounded by Proposition~\ref{pr1}(c). Since $\bigcap _iC_P(Q_i)=1$, the derived length of $P$ is $r$-bounded, and therefore the $p$-length of $\bar G$ is $r$-bounded by the theorems of Hall and Higman \cite[Theorem~A]{ha-hi} for $p\ne 2$, and of Berger and Gross \cite{ber-gro} and
Bryukhanova~\cite{bry81} for $p=2$

Thus, the Fitting height $h(G)$ is $r$-bounded. We now use induction on $h(G)$  to prove that ${\bf r}(\g _{\infty}(G))$ is $r$-bounded. Clearly, we only need to consider the case $G=F_2(G)$. Then by Lemma~\ref{l-f2} we have $\g _{\infty}(G)=\prod _q[F_q,H_{q'}]$, where $F_q$ is a Sylow $q$-subgroup of $F(G)$ and $H_{q'}$ is a Hall $q'$-subgroup of $G/C_G(F_q)$. The rank of each subgroup $[F_q,H_{q'}]$ is $r$-bounded by Proposition~\ref{pr1}(b). Therefore the rank of $\g _{\infty}(G)$ is also $r$-bounded.
\ep

\section{Nonsoluble groups}\label{s-nons}

In the interests of brevity, with a slight abuse of terminology, we say that in a finite group $G$ {\it all Engel sinks are of rank $r$} if for every $g\in G$ the Engel sink ${\mathscr E}(g)$ generates a subgroup of rank at most $r$.

\bl
 \label{l-simple}
 Let $G$ be a finite nonabelian simple group with all Engel sinks of rank~$r$. Then $G$ has $r$-bounded rank.
\el

\bp We can assume that $G$ is either an alternating group or a group of Lie type. For an alternating group $G=A_n$ it is easy to see that $n$ is $r$-bounded. Indeed, every finite group of order $m$ can be embedded in $A_n$ with $m$-bounded $n$, and there are groups of $r$-bounded order with ${\bf r}(\langle{\mathscr E}(g)\rangle)=r+1$ for some element $g$. (For example, in the semidirect product $A\langle b\rangle$ of an elementary abelian $3$-group $A$ of rank $r+1$ and its group of automorphisms $\langle b\rangle$ of order $2$ that acts fixed-point-freely, we have ${\bf r}(\langle{\mathscr E}(b)\rangle)=r+1$, and $|A\langle b\rangle|=2\cdot 3^{r+1}$.)

So let $G$ be a finite simple group of Lie type $G=L_n(\F _{p^k})$ of degree $n$ over a field of order $p^k$, where $p$ is a prime. It is sufficient to show that both $n$ and $k$ are $r$-bounded. Indeed, then the linear covering group $\hat G=\hat L_n(\F _{p^k})$ in its natural representation of dimension $n$ over $\F _{p^k}$ can be regarded as a linear group of $r$-bounded dimension $nk$ over $\F _p$. Therefore the rank of a Sylow $p$-subgroup of $\hat G$ is $r$-bounded by Lemma~\ref{l-gmh}, and the rank of a Sylow $t$-subgroup for $t\ne p$ is $r$-bounded by Lemma~\ref{l-r-coprime}. Then the rank of $\hat G$ is bounded by Lemma~\ref{l-kov}.

For $G=L_n(\F _{p^k})$, to obtain a bound for $k$ in terms of $r$, it is sufficient to show that $G$ has an element $g$ such that ${\mathscr E}(g)$ contains a subgroup isomorphic to the additive group of the field $\F _{p^k}$. This follows from the well-known facts about simple groups of Lie type. To be specific, one of the ways to show this is to use the fact that $G=L_n(\F _{p^k})$ either contains a subgroup isomorphic to $SL_2(p^k)$ or $PSL_2(p^k)$, or is a Suzuki group (over the field $\F _{p^k}$). For example, even stronger statements are proved in \cite{Liebeck--Nikolov--Shalev}. In $SL_2(p^k)$, put
$$g=\begin{pmatrix} \zeta ^{-1}&0\\0&\zeta \end{pmatrix},$$
where $\zeta $ is a nontrivial $p'$-element of the multiplicative group of the field $\F _{p^k}$ such that $\zeta ^2\ne 1$ (the latter condition can always be satisfied for $k>1$). This element normalizes and acts fixed-point-freely on the abelian $p$-subgroup of upper-triangular matrices $$T=\left\{\left.\begin{pmatrix} 1&a\\0&1\end{pmatrix}\right| a\in \F _{p^k}\right\},$$
which is isomorphic to the additive group of $\F _{p^k}$. Then $T\subseteq {\mathscr E}(g)$ by Lemma~\ref{l0}. In the quotient $PSL_2(p^k)$ of $SL_2(p^k)$ by the centre, the image of $T$ is isomorphic to $T$. Finally, the case of $G$ being a Suzuki group is dealt with in similar fashion, by considering the action of a diagonal $2'$-element on a Sylow $2$-subgroup. Thus, $k$ is $r$-bounded.

For $G=L_n(\F _{p^k})$, to obtain a bound for $n$, it suffices to consider the Weyl subgroup, which, for large $n$, contains a subgroup isomorphic to a symmetric group of large degree, which in turn contains Engel sinks of large rank, as explained at the beginning of the proof.

 As explained above, bounds in terms of $r$ for $n$ and $k$ imply that the rank of $G=L_n(\F _{p^k})$ is $r$-bounded.
\ep

\bl \label{l-gen}
Given a prime $p$, any non-abelian finite simple group $T$ of rank~$s$ can be generated by $s$-boundedly many $p'$-elements.
\el

\bp
If $p\ne 2$, then we can use Guralnick's result \cite[Theorem~A]{gur2} that $T$ is generated by an involution and a Sylow $2$-subgroup, which has rank at most $s$ by hypothesis. If $p=2$, we can use King's result \cite{king} that $T=\langle i,a\rangle$, where $|i|=2$ and $|a|$ is an odd prime; then $T=\langle a, a^i\rangle$, since this is an $a$-invariant and $i$-invariant subgroup, which is therefore normal.
\ep

\bpr \label{p-fs}
Let $G$ be a finite group such that $G=[G,G]$ and $G/F(G)$ is a non-abelian simple group. Suppose that all Engel sinks of $G$ have rank $r$. Then $G$ has $r$-bounded rank.
\epr

\bp
By Lemma~\ref{l-simple} the quotient $G/F(G)$ has $r$-bounded rank.
Thus, we need to show that the rank of $F(G)$ is $r$-bounded. It suffices to show that the rank of each Sylow $p$-subgroup of $F(G)$ is $r$-bounded. Considering the quotient of $G$ by the Hall $p'$-subgroup of $F(G)$, we can assume that $F(G)$ is a $p$-group.

Using Lemma~\ref{l-gen} we write $G/F(G)=\langle\bar a_1,\dots , \bar a_k\rangle$, where all the $\bar a_i$ are $p'$-elements, while $k$ is $r$-bounded.
We can choose some $p'$-elements $a_1,\dots ,a_k$ that are preimages of $\bar a_1, \dots , \bar a_k$ in $G$. Let  $S=\langle a_1,\dots ,a_k\rangle$; then $G=F(G)S$.

We have 
$$
[F(G),S]=[F(G),a_1]\cdots [F(G),a_k]
$$
by Lemma~\ref{l-prod}. The rank of each of the normal subgroups $[F(G),a_1]$ is at most $r$ by Lemma~\ref{l0}. Since $k$ is $r$-bounded, the rank of $[F(G),S]$ is $r$-bounded.

Therefore, factoring out $[F(G),S]$, we can assume that $[F(G),S]=1$. Then $S$ is a normal subgroup of $G$. Since $G=G'$ and $F(G)$ is nilpotent, it follows that  $G=S$. 

Under our assumption $[F(G),S]=1$, the group $G=S$ is a perfect group that is a central extension of the finite simple group $G/F(G)$. Hence $F(G)$ is isomorphic to a subgroup of the Schur multiplier of the simple group $G/F(G)$. Therefore $F(G)$ has rank at most $3$, as follows from the classification.
\ep

We are now ready to prove the main result in the general case.

\bt
If $G$ is a finite group with all Engel sinks of rank $r$, then $\g _{\infty} (G)$ has $r$-bounded rank.
\et

\bp
By Theorem~\ref{t-sol} applied to the soluble radical $R(G)$, the rank of $\g _{\infty}(R(G))$ is $r$-bounded. Factoring out $\g _{\infty}(R(G))$ we can assume that the soluble radical is nilpotent: $R(G)=F(G)$. Let the socle of $G/F(G)$ be $S_1\times \cdots \times S_m$, where the $S_i$ are non-abelian simple groups. We claim that the number of factors $m$ is at most $r$. Indeed, by the Feit--Thompson theorem and Frobenius' theorem, each $S_i$ has a nontrivial $2$-subgroup $T_i$ normalized but not centralized by a $2'$-element $a_i\in S_i$. Then $\langle {\mathscr E}(a_1a_2\cdots a_m)\rangle$ has rank at least $m$, so $m\leq r$.

Let $\bar G=G/F(G)$ and let $N=\bigcap N_{\bar G}(S_i)$. Since $m\leq r$ and the quotient $\bar G/N$ acts faithfully by permutations of the factors $S_1,\dots, S_m$, the order $|\bar G/N|$ is $r$-bounded. By the usual argument, since $\bigcap C_{\bar G}(S_i)=1$, the group $N$ embeds in the direct product of $m$ almost simple groups that are extensions of the $S_i$ by outer automorphisms. By Lemma~\ref{l-simple}, each $S_i$ has $r$-bounded rank. It follows from the classification that the order of the outer automorphism group of every $S_i$ is $r$-bounded. Indeed, alternating groups of $r$-bounded rank have $r$-bounded order. For a group of Lie type $L_n(p^f)$, the order $|{\rm Out}\,L_n(p^f)|$ is bounded in terms of $n$ and $f$, and both these parameters are $r$-bounded, as we saw in Lemma~\ref{l-simple}. Hence $|N/(S_1\times\dots\times S_m)|$ is $r$-bounded.

As a result, the composition length of $G/F(G)$ is $r$-bounded. We now complete the proof by induction on this composition length. Let $G_1$ be a normal subgroup of $G$ containing $F(G)$ with $G/G_1$ simple (abelian or non-abelian). By the induction hypothesis, $\g _{\infty }(G_1)$ has $r$-bounded rank. By passing to $G/\g _{\infty }(G_1)$ we can assume that
$G/F(G)$ is simple. If it is abelian, then Theorem~\ref{t-sol} applies to $G$. If it is non-abelian, then Proposition~\ref{p-fs} applies to $\g _{\infty }(G)$.
\ep

\section*{Acknowledgements}
The first
author was supported by the Russian Science Foundation, project no. 14-21-00065,
and the second
by FAPDF and CNPq-Brazil. The first author thanks the University of Brasilia for the hospitality that he enjoyed during his visit to Brasilia.

\end{document}